\documentclass[11pt, english]{article}

\usepackage[margin= 2 cm,bottom=19mm, top= 18mm]{geometry}

\usepackage[square,sort,comma,numbers]{natbib}
\usepackage{amsthm}
\usepackage{amsmath}
\usepackage{amssymb}
\usepackage[nodisplayskipstretch]{setspace}

\usepackage{mathtools}
\usepackage{graphicx}
\graphicspath{ {./images/} }
\usepackage[hidelinks]{hyperref}
\usepackage{cleveref}
\usepackage{hyperref}
\usepackage{enumitem}
\usepackage{framed}
\usepackage{subcaption}
\usepackage{xspace}
\usepackage{thmtools} 
\usepackage{thm-restate}
\usepackage{tikz-feynman}
\usepackage{thmtools}
\usepackage{thm-restate}

\usepackage{floatrow}

\theoremstyle{plain}

\newtheorem*{thm*}{Theorem}
\newtheorem{thm}{Theorem}[section]
\Crefname{thm}{Theorem}{Theorems}

\newtheorem*{lem*}{Lemma}

\Crefname{lem}{Lemma}{Lemmas}

\newtheorem*{claim*}{Claim}
\newtheorem{claim}[thm]{Claim}
\crefname{claim}{Claim}{Claims}
\Crefname{claim}{Claim}{Claims}

\Crefname{prop}{Proposition}{Propositions}

\Crefname{remar}{Remark}{Remarks}

\newtheorem{cor}[thm]{Corollary}
\crefname{cor}{Corollary}{Corollaries}

\newtheorem*{conj*}{Conjecture}

\crefname{conj}{Conjecture}{Conjectures}

\Crefname{qn}{Question}{Questions}

\newtheorem{obs}[thm]{Observation}
\Crefname{obs}{Observation}{Observations}

\Crefname{ex}{Example}{Examples}

\theoremstyle{definition}

\Crefname{prob}{Problem}{Problems}

\Crefname{defn}{Definition}{Definitions}

\theoremstyle{remark}

\renewenvironment{proof}[1][]{\begin{trivlist}
\item[\hspace{\labelsep}{\bf\noindent Proof#1.\/}] }{\qed\end{trivlist}}

\newcommand{\remove}[1]{}

\newcommand{\G}{\mathcal{G}}

\title{\vspace{-1 cm}
Long induced paths in expanders}
\date{}

\crefname{enumi}{observation}{parts}

\newcommand{\lL}{\lambda}

\newcommand{\GG}{\Gamma}

\newcommand{\TT}{\Theta}

\newcommand{\mb}[1]{\mathbb{#1}}
\newcommand{\sub}{\subseteq}
\newcommand{\sm}{\setminus}

\newcommand{\es}{\emptyset}

\author{
Nemanja Dragani\'c\thanks{
Mathematical Institute, University of Oxford, UK. Research supported by SNSF project 217926.\\
\emph{email}: \textbf{nemanja.draganic@maths.ox.ac.uk}.
}
\and
Peter Keevash\thanks{Mathematical Institute, University of Oxford, UK. Research supported by ERC Advanced Grant 883810.\\
\emph{email}: \textbf{keevash@maths.ox.ac.uk}.
}}
\begin{document} 
\maketitle
\begin{abstract}
We prove that any bounded degree regular graph with sufficiently strong spectral expansion
contains an induced path of linear length.
This is the first such result for expanders, strengthening an analogous result 
in the random setting by Dragani\'c, Glock and Krivelevich.
More generally, we find long induced paths in sparse graphs
that satisfy a mild upper-uniformity edge-distribution condition.
\end{abstract}

\section{Introduction}
The Longest Induced Path problem concerns finding a largest subset of vertices $S$ in a graph $G$ 
such that the induced subgraph $G[S]$ is a path. 
The decision problem of determining whether there exists an induced path of a specified length is NP-complete. It is difficult and interesting even on very specific graphs $G$. 
For example, the Snake-in-the-box Problem \cite{abbott1988snake,zemor1997upper,mutze2022combinatorial}, 
motivated by the theory of error-correcting codes  \citep{kautz1958unit},
asks for the longest induced path in the hypercube graph.
It also features in the study of Detour Distance  \citep{chartrand1993detour}
and the spread of information in social networks~\cite{gavril2002algorithms}. 
To illustrate the latter connection, we imagine that some person in a social network
shares some information with their connections (e.g.~by posting it on their profile),
who then might  share it with their connections, resulting to an information cascade,
in which the longest induced path represents the longest existing route for information transmission. 

The classical literature on Random Graphs considers many problems on finding induced subgraphs,
starting from several independent papers \cite{bollobas1976cliques, grimmett1975colouring, matula1976largest}
in the 1970's calculating the asymptotic independence number of $G(n,p)$ for fixed $p$ and large $n$.
This was later extended by Frieze \cite{frieze1990independence} to $p=c/n$ for large constant $c$
(as noted in~\cite{cooley2021large}, this extends to all $p\geq c/n$). 
Erdős and Palka \cite{erdos1983trees} initiated the study of induced trees in $G(n, p)$,
which developed a large literature \cite{de1986induced,frieze1987large,kuvcera1987large,luczak1988maximal}
before its final resolution  by de la Vega~\cite{de1996largest}, showing that 
the size of the largest induced tree matches the asymptotics found earlier for the independence number:
it is $\sim2 \log_q(np)$ for all $p\geq c/n$ for large constant $c$. 
The Longest Induced Path problem in $G(n,p)$ is also classical  \cite{luczak1993size,suen1992large},
and was recently resolved asymptotically by Dragani\'c, Glock and Krivelevich \cite{draganic2022largest}.

We extend this line of research to spectral expanders in the following result.
To set up notation, let $G$ be a $d$-regular graph on $n$ vertices
whose adjacency matrix has eigenvalues $\lL_1 \ge \dots \ge \lL_n$.
Following Alon, we call $G$ an $(n,d,\lL)$-graph with $\lL=\max\{\lL_2,-\lL_n\}$.

\begin{thm}\label{thm:ndlambda}
Let $G$ be an $(n,d,\lambda)$-graph with $\lambda<d^{3/4}/100$ and $d<n/10$. 
Then $G$ contains an induced path of length $\frac{n}{64d}$.
\end{thm}

One may compare Theorem \ref{thm:ndlambda} with an analogous result in the random setting
due to Dragani\'c, Glock and Krivelevich~\cite{draganic2022short}.
Their result is stated in a pseudorandom setting
but essentially concerns random graphs,
due to its strong assumptions on edge distribution
(sublinear sets have average degree $<3$).

One may also interpret Theorem \ref{thm:ndlambda} 
within the theory of Graph Sparsity (see  \cite{nevsetvril2012sparsity}),
where one general problem considers some graph class $\G$
and asks for the optimal function $f_\G:\mb{N}\to\mb{N}$ such that
any graph in $\G$ with a path of length $k$
must contain an induced path of length $f_\G(k)$.
For example, if $\G$ is the class of $C$-degenerate graphs 
then Nešetřil and Ossona de Mendez \cite{nevsetvril2012sparsity}
showed that $f_\G(k) = \Omega(\log \log k)$,
whereas Defrain and Raymond~\cite{defrain2024sparse}
showed that $f_\G(k) = O(\log\log k)^2$ even for $C=2$.
Letting $\G$ be the class of spectral expanders as in Theorem  \ref{thm:ndlambda},
also assuming constant average degree,
it is well-known that any such $G$ 
contains a path of length $n-o(n)$,
so Theorem  \ref{thm:ndlambda} shows that $f_\G(k)=\TT(k)$.

Our next result is in the same spirit as Theorem \ref{thm:ndlambda},
replacing spectral expansion with an edge distribution condition,
which is a mild upper-uniformity condition. 
Here and throughout the paper we write
\[ \GG(X) = \GG_G(X) = \{ v \in V(G): N(v) \cap X \ne \es \}
\ \text{ and } \
e(X,Y)=e_G(X,Y) = \sum_{x \in X, y \in Y} 1_{xy \in E(G)}.\]

\begin{thm}\label{thm:upperuniform}
Let $G$ be a graph on $n$ vertices with minimum degree $d \ge 2^8$.
Suppose for some $C>1$ that for all $X,Y \sub V(G)$
of sizes $|X|=\frac{n}{2^4 Cd}$ and $|Y|= \frac{n}{2^8 C}$ we have 
$e(X,Y)< 2^5 C|X||Y|\frac{d}{n}$ and $e(\Gamma(X),Y) < C|X||Y|\frac{d^2}{n}$.
Then $G$ contains an induced path of length $\frac{n}{2^5 Cd}$.
\end{thm}

The numerical constants in Theorem \ref{thm:upperuniform} are quite flexible;
we have selected a convenient choice that is whp satisfied in a subgraph $G'\subseteq G(n,d/n)$ of minimum degree $\delta(G')\geq 0.9d$ and size $N=|V(G')|\geq 0.9n$, for $d \ge 2^{20}$ and $C=100$, say.
Superimposing $2^{20} N/d$ vertex-disjoint $2^{-20} d$-cliques 
has very little effect on the upper-uniformity condition,
and clearly ensures that there is no induced path of length~$>2^{20} N/d$,
so Theorem \ref{thm:upperuniform} is tight up to the constant factor.

Next we will state our main result, which will easily imply the two results stated above.
For a third application in Ramsey Theory
(a modest improvement on the multicolour induced-size Ramsey numbers of paths)
we consider a more general problem where we are given graphs $G' \sub G$ 
and look for a long path in $G'$ that is induced in $G$.

\begin{thm} \label{thm:DFSnew}
Suppose $G$ and $G'$ are graphs on the same vertex set $V$ with $G' \sub G$.
Let $d$ be the minimum degree of $G'$ and 
$s_1,s_2,\ell$ be positive integers with $s_1+s_2+\ell<|V|$.
Suppose that
\begin{enumerate}[label=(\arabic*)]
    \item $e_{G'}(X,Y)< \frac{d}{4}s_1$ for all $X,Y\subseteq V$ with $|X| = s_1$ and $|Y| = \ell+s_1+s_2$.
    \item $e_{G'}(\Gamma_G(X),Y)< \frac{d}{4}s_2$ for all disjoint $X,Y\subseteq V$ with $|X| = \ell+s_1$ and $|Y| = s_2$.
\end{enumerate}
Then $G'$ contains a path of length $\ell$ which is induced in $G$, and can be found in time $O(e(G))$.
\end{thm}

Our results are based upon a new version of the DFS graph search algorithm, 
tailored for applications on graphs with a weak local sparsity condition. 
The algorithm runs in time $O(e(G))$, which is clearly best possible.
We describe this algorithm and prove our main theorem in  \Cref{sec:algorithm};
the applications will be deduced in \Cref{sec:applications}.

\medskip

\noindent \textbf{Notation.} Besides the notation $\GG(X)$ and $e(X,Y)$ mentioned above,
we also write $N(X) = \GG(X) \sm X$ for the external neighbourhood of $X$ in $G$.
We systematically avoid rounding notation for clarity of presentation whenever it does not impact the argument.

\section{The algorithm and its analysis}\label{sec:algorithm}

\begin{figure}

 \textbf{Algorithm} 

\begin{flushleft}
We initialise $T:=V$, $P:=\emptyset$, and $S_1:=\emptyset$ and $S_2:=\emptyset$. \\
We carry out our algorithm in rounds, stopping when $S_1\cup S_2=V$. 
\end{flushleft} 

 \textbf{Round}

\begin{enumerate}[noitemsep]
    \item If $P=\emptyset$, take the next vertex $v$ from $T$ according to $\sigma$, remove it from $T$ and push it to $P$. Otherwise, let $v$ be the vertex from the top of the stack of $P$. 
    \item\label{step:dispose to S_2} If at least $\frac{d_{G'}(v)}{2}$ of the neighbours of $v$ in $G'$ are inside $N_{G}(P-v)$, then remove $v$ from $P$ and add it to $S_2$.
    \item Else, if at least $\frac{d_{G'}(v)}{2}$ of the neighbours of $v$ in $G'$ are inside $P\cup S_1\cup S_2$, then remove $v$ from $P$, and add it to $S_1$.
    \item Otherwise, there is a vertex in $N_{G'}(v)\setminus (S_1\cup S_2\cup P\cup N_G(P-v))\subseteq T$. Push this vertex to $P$, and remove it from $T$.
\end{enumerate}

\caption{The algorithm used in the proof of \Cref{thm:DFSnew}.}  
\label{fig1}
\end{figure}

In this section we describe and analyse our algorithm (see \Cref{fig1}), 
thus proving our main result, Theorem \ref{thm:DFSnew}.
As in the statement, we consider graphs $G,G'$ on $V$ with $G' \sub G$.
Let $\sigma$ be an arbitrary ordering of $V$.
Throughout the algorithm we maintain a partition of $V$
into four sets $T$, $P$, $S_1$ and $S_2$,
where $T$ will be the set of vertices which have not yet been considered,
$P$ will be the set of vertices in the current path, 
and $S_1$ and $S_2$ will be two different sets of discarded vertices. 
The vertices of $P$ are kept in a stack so that following the order in the stack
gives a path in $G'$ which is induced in $G$.

In the proof of \Cref{thm:DFSnew}, we will rely on the following observations.
We omit their proofs, as parts  \ref{p:P is path} to \ref{p:grows by 1} are self-evident,
and the complexity analysis for  \ref{p:complexity} 
is similar in spirit to that of the standard DFS algorithm.
Only the key property in  \ref{p:S_2 immediate} requires a little thought.
The key point is that whenever $v$ is added to $S_2$
the path $P-v$ is the same as when $v$ was added to $P$,
so if it satisfies the condition for joining $S_2$ at any point
in the algorithm then it does so immediately when added to $P$.

\begin{obs}
    The following hold during the execution of the algorithm.
\begin{enumerate}[label=(\Alph*)]
    \item\label{p:P is path} The vertices in $P$ form a path in $G'$ which is induced in $G$.
    \item\label{p:vertex stays in PS1S2} If a vertex $v$ lands in $P\cup S_1\cup S_2$, then it stays in this set until the algorithm terminates.
    \item\label{p:end state} When the algorithm terminates, we have $S_1\cup S_2=V$.
    \item\label{p:grows by 1} In each round, at most one vertex is added to $S_1\cup S_2$.
    \item\label{p:S_2 immediate} \textbf{(key property)} If a vertex is added to $S_2$, it is added there immediately in the round after the round in which it appears in $P$.
    \item\label{p:complexity} The algorithm can be implemented in time $O(e(G))$.
\end{enumerate}
\end{obs}
     
Now we analyse the algorithm, thus proving our main result.

\begin{proof}[ of Theorem \ref{thm:DFSnew}]
We consider the algorithm above applied to $G$ and $G'$. 
Suppose for a contradiction that it does not find the required path.
By \Cref{p:P is path} at each step of the algorithm we have $|P|<\ell$.
Consider the round after which $|S_1|=s_1$ or $|S_2|=s_2$ for the first time; 
this has to occur by \Cref{p:end state} and \Cref{p:grows by 1}. 
We distinguish two cases, according to whether $|S_1|=s_1$ or $|S_2|=s_2$.
 
The first case is that $|S_1|=s_1$, and so $|S_2|<s_2$. 
Each vertex added to $S_1$ has at  least $\frac{d}{2}$ of its $G'$-neighbours 
in $P\cup S_1\cup S_2$ when it is added, and these neighbours
remain in $P\cup S_1\cup S_2$ by~\Cref{p:vertex stays in PS1S2}. 
Thus at the round under consideration, 
we have $|P \cup S_1 \cup S_2| \le \ell+s_1+s_2$
and $e_{G'}(S_1,S_1\cup S_2\cup P)\geq s_1\frac{d}{4}$, 
a contradiction with the first condition of the theorem.

The second case is that $|S_2|=s_2$, and so  $|S_1|<s_1$. 
Each vertex added to $S_2$ had at least $\frac{d}{2}$ of its $G'$-neighbours 
in $\GG(P)$ at the time when it was added, and the path vertices
at that time are either still on the path or were added to $S_1$,
not $S_2$, by the key property~\Cref{p:S_2 immediate}.
Thus each vertex in $S_2$ has at least $\frac{d}{2}$ of its $G'$-neighbours 
in $\GG(S_1 \cup P)$, so $e_{G'}(S_2,\GG(S_1 \cup P)) \ge s_2 \frac{d}{4}$.
However, $|P\cup S_1|\leq \ell+s_1$, so this contradicts the second condition of the theorem.
\end{proof}

\section{Applications}\label{sec:applications}

Now we will apply our main result, Theorem \ref{thm:DFSnew},
to prove the applications mentioned above,
i.e.~Theorems \ref{thm:ndlambda} and \ref{thm:DFSnew},
and an induced Ramsey result to be discussed below in Section \ref{sec:Ramsey}.

First we prove our result on long induced paths in expanders.
We use the following two well-known properties of $(n,d,\lL)$-graphs
that can be found e.g.~in the survey \cite{hoory2006expander}.

\begin{enumerate}[nosep]
\item[(E1)] Expander Mixing Lemma \cite[Lemma 2.4]{hoory2006expander}:
$\Big|e(A,B)-|A||B|\frac{d}{n}\Big|\leq \lambda\sqrt{|A||B|}$ for any $A,B \sub V(G)$.
\item[(E2)] Simplified Alon-Boppana Bound \cite[Claim 2.8]{hoory2006expander}:
 $\lambda^2 \ge d \cdot \frac{n-d}{n-1}$.
\end{enumerate}

\begin{proof}[ of Theorem \ref{thm:ndlambda}]
Suppose $G$ is a $(n,d,\lambda)$-graph with $\lambda<d^{3/4}/100$ and $d<n/10$. 
We will apply  Theorem \ref{thm:DFSnew} with $G'=G$, 
$\ell=s_1=\frac{n}{64d}$ and $s_2=\frac{\lambda^2}{d^2}n$. We note that $s_1+s_2+\ell<n$.
By Property (E2) we have $s_2 \ge \frac{n}{d} \cdot \frac{n-d}{n-1} > s_1$.

For condition (1), consider any $X,Y \sub V(G)$ with $|X| = s_1$ and $|Y| = \ell+s_1+s_2 < 3s_2$.
By Property (E1) we have
\[ e(X,Y)\leq s_1\cdot 3s_2\cdot \frac{d}{n}+\lambda \sqrt{s_1\cdot 3s_2}
\leq  \frac{\lambda^2}{d^2}n+\frac{\lambda^2}{d^{3/2}}n<\frac{n}{1000}<\frac{d}{4}s_1.\]
For condition (2), consider any $X,Y \sub V(G)$ with $|X| = 2s_1$ and $|Y| = s_2$.
As $G$ is $d$-regular we have $|\Gamma(X)|\leq d|X|$, so by Property (E1) we have
\[ e(\Gamma(X),Y)\leq |\Gamma(X)||Y|\frac{d}{n}+\lambda\sqrt{|\Gamma(X)||Y|}
\leq \frac{\lambda^2}{32d}n+ \frac{\lambda^2}{\sqrt{32}d}n\leq s_2\frac{d}{4}. \]
Since all conditions of Theorem \ref{thm:DFSnew} are satisfied,
there is an induced path of length $\ell=\frac{n}{64d}$.
\end{proof}

Now we prove  our result on long induced paths in graphs
satisfying an upper-uniformity edge-distribution condition.

\begin{proof}[ of Theorem \ref{thm:upperuniform}]
Let $G$ be a graph on $n$ vertices with minimum degree $d \ge 2^8$.
Suppose for some $C>1$ that for all $X,Y \sub V(G)$
of sizes $|X|=\frac{n}{2^4 Cd}$ and $|Y|= \frac{n}{2^8 C}$ we have 
$e(X,Y)< 2^5 C|X||Y|\frac{d}{n}$ and $e(\Gamma(X),Y) < C|X||Y|\frac{d^2}{n}$.
To prove the theorem, it suffices to show that the conditions of  \Cref{thm:DFSnew}
are satisfied with $\ell=s_1=\frac{n}{2^5 Cd}$ and $s_2=\frac{n}{2^9 C}$. First note that $s_1+s_2+\ell<n$.

For condition (1), consider any $X,Y \sub V(G)$ with $|X| = s_1$ and $|Y| = \ell+s_1+s_2$.
Enlarging $X$ and $Y$ to sizes $2s_1 = \frac{n}{2^4 Cd}$ and $2s_2 = \frac{n}{2^8 C}$, we have
$$e(X,Y)<2^5C|X||Y|\frac{d}{n}\leq 2^5C \cdot 2s_1 \frac{n}{2^8 C}\frac{d}{n} = s_1d/4,$$
as required.
For condition (2), consider any $X,Y \sub V(G)$ with $|X| = \ell+s_1$ and $|Y| =s_2$. 
Enlarging $Y$ to size $2s_2 = \frac{n}{2^8 C}$, we have
$$e(\Gamma(X),Y)<C|X||Y|\frac{d^2}{n}\leq C\frac{n}{2^4 Cd} 2s_2\frac{d^2}{n}< ds_2/4,$$
as required. This completes the proof.
\end{proof}

\subsection{Induced size-Ramsey number of paths} \label{sec:Ramsey}

Here we consider a problem on induced size-Ramsey numbers,
combining two well-studied extensions of the classical graph Ramsey problem,
in which one colours some `host' graph $G$ and seeks 
a monochromatic copy of some `target' graph $H$.
In the induced Ramsey problem, one seeks monochromatic induced copies of $H$
and in the size-Ramsey problem one aims to minimise the size $e(G)$ of $G$.
Induced size-Ramsey problems combine both of these features:
the $k$-colour induced size-Ramsey number $\hat{r}^k_{ind}(H)$ 
is the smallest integer $m$ such that there exists a graph $G$ on $m$ edges 
such that every $k$-colouring of the edges of $G$ contains a monochromatic copy of $H$ that is an induced subgraph of $G$. 
The main open problem in this direction is an old conjecture of Erd\H{o}s
that for any graph $H$ on $n$ vertices one has $\hat{r}^2_{ind}(H) \le 2^{O(n)}$.
For more background we refer the reader to the survey \cite{conlon2015recent}
and recent papers \cite{draganic2022size,bradavc2023effective}.

The case when $H=P_n$ is a path has a large literature in its own right.
Haxell, Kohayakawa and Łuczak~\cite{haxell1995induced} showed that $\hat{r}^k_{ind}(P_n)=O_k(n)$,
strengthening the analogous result of  Beck \cite{beck1983size} for the
(not necessarily induced) size-Ramsey number, 
which in itself was a $\$$100 problem of Erd\H{o}s.
While these results establish the order of magnitude as $O_k(n)$,
they do not say much about the implicit constant,
particularly in~\cite{haxell1995induced} due to the use of the regularity lemma.
Even for the two-colour size-Ramsey number of $P_n$,
there is a substantial constant factor gap 
between the best known lower bound \cite{bal2019new}
and upper bound \cite{dudek2017some}.
For the multicolour induced size-Ramsey number,
a significant recent improvement on~\cite{haxell1995induced}
by Dragani\'c, Glock and Krivelevich~\cite{draganic2022short}
gives $\hat{r}^k_{ind}(P_n)=O(k^3\log^4 k)n$.
We will improve this to $\hat{r}^k_{ind}(P_n)=O(k^3\log^2k)n$.

\begin{thm}\label{Gnp contains path}
    For $c=10^{5}$ and for all large enough $k\in \mathbb N$ the following holds. Let $G\sim G(nk,\frac{c\log k}{n})$. 
    Then with high probability, for every $k$-colouring of the edges of $G$, there exists a monochromatic path of length $\frac{n}{c^3k\log k}$ which is induced in $G$.
\end{thm}
\begin{proof}
 Let $p=\frac{c\log k}{n}$ and let $G\sim G(nk,p)$. We first show that for $\ell=s_1=s_2=\frac{n}{c^3k\log k}$ the following hold with high probability
\begin{enumerate}[label=(\alph*)]
    \item \label{local-density1} $e_G(X,Y)< s_1 \frac{c\log k}{16}=\frac{n}{16kc^2}$ for all  $X\subseteq V(G)$ with $|X| = s_1$ and $|Y| = \ell+s_1+s_2$.
    \item\label{local-density} $e_G(\Gamma_G(X),Y)\leq s_2 \frac{c\log k}{16}=\frac{n}{16kc^2}$ for all disjoint $X,Y\subseteq V(G)$ with $|X| = |Y| = 2\ell$.
    \item\label{edge count for last part} $e_G(X)\leq \frac{n}{4c^2k}$ for $|X|=s_1+s_2+\ell$.   
\end{enumerate}

For~\ref{local-density1}, for any such $X,Y$ we have $\mb{E} e(X,Y) = p|X||Y| = \frac{3n}{c^5 k^2 \log k}$,
so by Chernoff the bound fails with probability at most $e^{-\frac{n}{kc^2}}$, say, and we can take a union bound 
over  $\binom{nk}{\ell}\binom{nk}{3\ell}\leq e^{\frac{10n}{kc^3}}$ choices for $X,Y$. Part~\ref{edge count for last part} is similar.
For~\ref{local-density}, we first note that similarly to~\ref{local-density1} we can assume for any such $X,Y$ that 
$e(X,Y) < \frac{n}{32kc^2}$, so it suffices to show $e(N(X),Y ) < \frac{n}{32kc^2}$, where $N(X)=\GG(X) \sm X$.
Also, by Chernoff we have $\mb{P}( e(X,V(G)) > 2n/c^2 ) < e^{-\frac{n}{10c^2}}$, 
so by a union bound we can assume $|N(X)| \le e(X,V(G)) \le 2n/c^2$. 
We note that $e(N(X),Y)$ is independent of the edges incident to $X$,
so conditional on $|N(X)| \le e(X,V(G)) \le 2n/c^2$, 
Chernoff gives $\mb{P}( e(N(X),Y ) > \frac{n}{32kc^2} ) <  e^{-\frac{n}{800kc^2}}$,
so we can take a union bound.

Let $G_0$ be the subgraph of $G$ consisting of the edges of the densest colour class. 
As whp $G$ has $(0.5-o(1))nck^2\log k$ edges, $G_0$ has at least $(0.5-o(1))nck\log k$ edges, 
so has average degree at least $c\log k/2$. Now, let $G'$ be a subgraph of $G_0$ 
obtained by arbitrarily removing vertices with degree less than $c\log k/4$ one by one until this is no longer possible. 
Note that removing vertices in this way can not decrease the average degree.
This process must stop with more than $s_1+s_2+\ell$ vertices remaining,
otherwise we reach a set of $s_1+s_2+\ell$ vertices spanning 
at least $(s_1+s_2+\ell)c\log k/4\geq\frac{n}{2c^2k}$ edges, 
but this contradicts~\ref{edge count for last part}.
To conclude, we apply Theorem~\ref{thm:DFSnew} to the graphs $G[V(G')]$ and $G'$
with parameters $s_1=s_2=\ell$ and $d= c\log k/4$, 
noting that \ref{local-density1} and \ref{local-density} 
give the required conditions of the theorem.
\end{proof}

Since $G(nk,\frac{c\log k}{n})$ whp has $\Theta(nk^2\log k)$ edges,
Theorem~\ref{Gnp contains path} immediately gives the following.

\begin{cor}
Multicolour induced size-Ramsey numbers of paths satisfy
$\hat{r}_{ind}^k(P_n)= O(nk^3\log^2 k)$.
\end{cor}

\noindent {\bf Acknowledgement.}
We thank an anonymous referee for reading the paper carefully.

\end{document}